\documentclass[12pt]{amsart}
\usepackage[T1]{fontenc}
\usepackage[utf8]{inputenc}
\usepackage{a4wide}
\usepackage{diagbox}
\usepackage{amssymb}
\usepackage{mathtools}
\usepackage{dsfont}
\def\be{\begin{equation}}
\def\ee{\end{equation}}

\usepackage{amsmath}
\usepackage{graphicx}
\usepackage{amssymb}
\usepackage{bbm}
\usepackage{amsthm}
\usepackage{geometry}
\usepackage{tikz-cd}
\usepackage{mathrsfs}
\usepackage[colorinlistoftodos]{todonotes}

\usepackage{hyperref}
\hypersetup{hypertex=true,colorlinks=true,linkcolor=red,anchorcolor=red,citecolor=cyan}

\newtheorem*{property*}{Property}

\newtheorem{conjecture}{Conjecture}
\newtheorem*{conjecture*}{Conjecture}

\newtheorem*{completeness*}{Completeness property}
\newtheorem*{theorem*}{Theorem}
\newtheorem{theorem}{Theorem}
\newtheorem{proposition}{Proposition}
\newtheorem*{proposition*}{Proposition}
\newtheorem{lemma}{Lemma}
\newtheorem{corollary}{Corollary}
\theoremstyle{remark}
\newcommand{\mms}{{\mathcal{S}}}

\newcommand{\nc}{\newcommand}

\newcommand{\cl}{{\mathbf Y_{\ell}}}

\newcommand{\PP}{\mathscr{P}}

\newcommand{\RR}{R}

\newcommand{\mme}{\mathrm{e}}
\newcommand{\mmi}{\mathrm{i}}

\newcommand{\M}{\mathcal{M}}

\newcommand{\N}{{\mathbb N}}

\newtheorem*{thm*}{Theorem A}
\newtheorem*{thm**}{Theorem B}
\newtheorem*{conj*}{Conjecture A}
\newtheorem*{conj**}{Conjecture B}
\nc{\supp}{\operatorname{supp}}
\nc{\Real}{\operatorname{Re}}
\nc{\Imag}{\operatorname{Im}}

\nc{\Hi}{{\mathscr{H}}^\infty} \nc{\Ht}{{\mathscr{H}}^2}
\nc{\Hone}{{\mathscr{H}}^1} \nc{\ol}{\overline} \nc{\bz}{\mathbf{z}}
\nc{\bw}{\mathbf{w}} \nc{\eps}{\varepsilon}

\DeclareMathOperator\dif{d\!}

\allowdisplaybreaks[3]

\begin{document}

\title[On derivatives of zeta and $L$-functions near the 1-line]
{On derivatives of zeta and $L$-functions near the 1-line}
\author{Zikang Dong}
\author{Yutong Song}
\author{Weijia Wang}
\author{Hao Zhang}
\address{}

\email{zikangdong@gmail.com}
\email{99yutongsong@gmail.com}
\email{weijiawang@amss.ac.cn}
\email{zhanghaomath@hnu.edu.cn}
\address[Zikang Dong]{School of Mathematical Sciences, Tongji University, Shanghai 200092, P. R. China}
\address[Yutong Song]{School of Mathematical Sciences, Tongji University, Shanghai 200092, P. R. China}
\address[Weijia Wang]{Morningside Center of Mathematics, Academy of Mathematics and Systems Science, Chinese Academy of
Sciences, Beijing 100190, P. R. China}
\address[Hao Zhang]{School of Mathematics, Hunan University, Changsha 410082, P. R. China}

\begin{abstract}
We study the conditional upper bounds and extreme values of derivatives of the Riemann zeta function and Dirichlet $L$-functions near the 1-line. Let $\ell$ be a fixed natural number. We show that,  if $|\sigma-1|\ll1/\log_2t$, then $|\zeta^{(\ell)}(\sigma+\mmi t)|$ has the same maximal order (up to the leading coefficients) as $|\zeta^{(\ell)}(1+\mmi t)|$ when $t\to\infty$.  The range $1-\sigma\ll1/\log_2t$ is wide enough, since we also show that $(1-\sigma)\log_2t\to\infty\;(t\to\infty)$ implies $\limsup_{t\to\infty}|\zeta^{(\ell)}(\sigma+\mmi t)|/(\log_2t)^{\ell+1}=\infty$. 
Similar results can be obtained for Dirichlet $L$-functions $L^{(\ell)}(\sigma,\chi)$ with $\chi\pmod q$.
\end{abstract}
\maketitle
\section{Introduction}
For any positive number $y$, define the set of $y$-friable numbers:
$$\mms(y):=\{n\in{\mathbb Z}\,|\,P_+(n)\le y\},$$ where $P_+(n)$ is the largest prime factor of $n$.
  Let $f$ be any arithmetic function, and $x$ large. Define the functions $\Psi(x,y)$ and $ \Psi(x,y; f)$ as
\begin{align*}
 \Psi(x,y):\, = \sum_{\substack{n \leq  x \\ n\in\mms(y)}}   1\,, \quad \quad \quad \Psi(x,y; f):\, = \sum_{\substack{n \leq x \\ n\in\mms(y)}}   f(n)\,.
\end{align*}
If $f(n)$ satisfies $|f(n)|\le1$ for any $n\in\mathbb N$, an important problem in analytic number theory is for how large $x$ we have $\sum_{n\le x}f(n)=o(x)$. When $f$ is chosen to be Dirichlet character, Granville and Soundararajan \cite{largeGS} proved the following celebrated asymptotic formula for large character sums. 

\begin{theorem*}[Theorem 2 \cite{largeGS}]\label{GRH_GS}
Let $\chi$ be any non-principal character $\pmod {q}$, and  assume  the Riemann Hypothesis for $L(s, \chi)$.
If $1\leq x \leq q$ and $y\geq \log^2q\log^{2} x (\log_2 q )^{12}$,  then
$$
\sum_{n\leq x} \chi(n) = \Psi(x,y;\chi)+O\biggl( \frac{\Psi(x,y)}
{(\log_2 q )^2}\biggr).
$$
Further
$$
\biggl| \sum_{n\leq x} \chi(n) \biggr| \ll
\Psi(x, \log^{2}q (\log_2 q )^{20}),
$$
and so the following estimate  holds 
\begin{align*}
\biggl |\sum_{n\leq  x} \chi(n)\biggr| = o(x) \,,   
\end{align*}
when $\log x/\log_2 q  \to \infty$ as $q\to \infty$.

\end{theorem*}
The study of large character sums can date back to 1918, when P\'olya and Vinogradov independently showed the non-trivial upper bound:
$$\sum_{n\le x}\chi(n)\ll \sqrt q\log q.$$ Further significant improvement was made subsequently by Montgomery and Vaughan \cite{MV77}, and Granville and Soundararajan \cite{GS07}.

When $f(n) = n^{-\mmi t}, \,\forall n \in \N$, we  write  $\Psi(x,y; t)$ in place of $\Psi(x,y; f)$. 
Recently, Yang \cite{Yangblms} generalized the above Granville-Soundararajan Theorem to the case of large zeta sums.  In \cite{zetasum}, part of the authors also study the Omega-results of zeta sums.
\begin{theorem*}[Theorem 5 \cite{Yangblms}]\label{RH_DY}
Assume  RH and let $T$ be sufficiently large.
If \,$2\leq x \leq T$, $ T + y + 3\leq t \leq  T^{1000}$  and 
$ y\geq \log^2 T\log^{2} x (\log_2 T)^{12}$,  then
\begin{align}\label{Approx_BySmooth}
\sum_{n\leq x}  \frac{1}{n^{\mmi t}} = \Psi(x,y;t)+O\biggl( \frac{\Psi(x,y)}
{(\log_2 T)^2}\biggr).
\end{align}
Further, for all $x \in [2,\, T] $ and $t \in [T + \log^{2}T (\log_2 T)^{15} ,\,  T^{1000}]$, we have 
\begin{align}\label{Upper_BySmooth}
    \biggl| \sum_{n\leq x}  \frac{1}{n^{\mmi t}}  \biggr| \ll
\Psi(x, \log^{2}T (\log_2 T)^{20}),
\end{align}
and for all $x \in [2,\, T]$ and $t \in [T + \log^{2}T (\log_2 T)^{15} ,\,  T^{1000}]$, we have the following estimate, 
\begin{align*}
\biggl |\sum_{n\leq  x} \frac{1}{n^{\mmi t}}\biggr| = o(x),
\end{align*}
when $\log x/\log_2 T \to \infty$ as $T\to \infty$.

\end{theorem*}
Also in \cite{Yangblms}, Yang used the above two theorems as powerful tools to develop the conditional upper bounds for the derivatives of the Riemann zeta functions  on the 1-line and Dirichlet $L$-functions at $s=1$. More precisely, he showed that GRH implies
\begin{align}\label{upperL}
\left|L^{(\ell)}(1, \chi) \right| \leq  \left( 2^{\ell + 1} Y_\ell + o(1) \right) \left(\log_2 q \right)^{\ell + 1}, \qquad \text{as } q \to \infty,\end{align}
and RH implies
\begin{align}\label{upperzeta}
\left| \zeta^{(\ell)}\left(1+\mmi t\right)\right| \leq  \left( 2^{\ell + 1}Y_\ell + o(1) \right) \left(\log_2 t \right)^{\ell + 1}, \qquad \text{as } t \to \infty,
\end{align}
where $Y_\ell= \int_0^{\infty} u^{\ell} \rho (u) \dif u$ is an absolute constant and $\rho(u)$ denotes the Dickman function. We refer to \cite{CC,CSo,LL} for the conditional upper bounds of the Riemann zeta function.

Granville and Soundararajan \cite{largeGS}, and Yang \cite{Yangblms} also made the following
conjectures separately, which are even stronger than the above two theorems.

\begin{conjecture}[Conjecture 1 \cite{largeGS}]\label{GSC}
There exists a constant $A>0$ such that
for any non-principal character $\chi$ (mod $q$), and for any
$1\leq  x\leq  q$ we have, uniformly,  
$$
\sum_{n\leq  x} \chi(n) = \Psi(x,y;\chi)+o(\Psi(x,y;\chi_0)),
$$
where $y= (\log q +\log^2 x) (\log_2 q )^A$.
\end{conjecture}
For zeta sums, Yang \cite{Yangblms} made the following analogous conjecture.
\begin{conjecture}[Conjecture 1 \cite{Yangblms}]\label{DYC}
There exists a constant $A>0$ such that
 for any
$1\leq  x \leq  T$, $2 T \leq  t \leq  5T$, we have, uniformly,  
$$
\sum_{n\leq  x}  \frac{1}{n^{\mmi t}} = \Psi(x,y;t)+o\left( \Psi(x,y)
\right),  \qquad \text{as } T \to \infty\,,
$$
where $y= (\log T +\log^2 x) (\log_2 T)^A$.
\end{conjecture}
Assuming these conjectures are true separately, Yang \cite{Yangblms} improved the upper bounds \eqref{upperL} and \eqref{upperzeta} by canceling the factor $2^{\ell+1}$:
Conjecture \ref{GSC} implies
\begin{align*}
\left|L^{(\ell)}(1, \chi) \right| \leq  \left(   Y_\ell + o(1) \right) \left(\log_2 q \right)^{\ell + 1}, \qquad \text{as } q \to \infty,
\end{align*}
and Conjecture \ref{DYC} implies
\begin{align*} \left| \zeta^{(\ell)}\left(1+\mmi t\right)\right| \leq  \left(   Y_\ell + o(1) \right) \left(\log_2 t \right)^{\ell + 1}, \qquad \text{as } t \to \infty,\end{align*}
These upper bounds are the best possible, since Yang also showed the existence of extreme values of the same quantities. We refer to \cite{DongNote,DY,Yangblms} for more details about this subject.

In this article, we will generalize the above results of Yang \cite{Yangblms} to the cases of  $	\left| \zeta^{(\ell)}\left(\sigma+\mmi t\right)\right| $ and $\left|L^{(\ell)}(\sigma, \chi) \right|$ when $\sigma$ is very close to 1. Our study is inspired by \cite{Daodao}, where Yang investigated the relation between $	\left| \zeta\left(\sigma+\mmi t\right)\right| $  and $	\left| \zeta\left(1+\mmi t\right)\right| $  as $\sigma\to1$.
Firstly, for the upper bounds on RH and GRH, we have the following two Theorems. 

\begin{theorem}\label{RH_upperBound}
Fix $\ell \in \mathbb{N}$ and let $A$ be any positive real number. Assuming  RH, then for $1-A/\log_2t<\sigma<1$, we have
	\begin{align*}
		\left| \zeta^{(\ell)}\left(\sigma+\mmi t\right)\right| \leq  \left( 2^{\ell + 1} C_\ell + o(1) \right) \left(\log_2 t \right)^{\ell + 1}, \qquad \text{as } t \to \infty,
	\end{align*}
 where $C_\ell =C_\ell(A)=\int_0^{\infty} \mme^{2Au}u^{\ell} \rho (u) \dif u$. 
\end{theorem}

\begin{theorem}\label{GRH_upperBound}
Fix $\ell \in \mathbb{N}$. Let $\chi$ be any non-principal character (mod $q$), and assume   GRH for $L(s, \chi)$. Then for $1-A/\log_2q<\sigma<1$, we have
	\begin{align*}
	\left|L^{(\ell)}(\sigma, \chi) \right| \leq \left( 2^{\ell + 1}C_\ell + o(1) \right) \left(\log_2 q \right)^{\ell + 1}, \qquad \text{as } q \to \infty,
	\end{align*}
 where $C_\ell$ is defined in Theorem \ref{RH_upperBound}.
\end{theorem}

In the opposite direction, we show the unconditional Omega-results for  $\zeta^{(\ell)}\left(\sigma+\mmi t\right)$ and $L^{(\ell)}(\sigma, \chi)$.
\begin{theorem}\label{thm3}
    Let $\sigma_A:=1-A/\log_2T$. Then we have
    $$\max_{t\in[T,2T]}\big|\zeta^{(\ell)}(\sigma_A+\mmi t)\big|\ge \left(D_\ell+o(1)\right)(\log_2T)^{\ell+1},$$
     where
    $$D_\ell=D_\ell(A):=\int_0^{\infty} \mme^{Au}u^{\ell} \rho (u) \dif u.$$
\end{theorem}

\begin{theorem}\label{thm4}
    Let $\sigma_A, D_\ell$ be defined as in Theorem \ref{thm3}. Then we have
    $$\max_{\chi\neq\chi_0({\rm mod}\; q)}\big|L^{(\ell)}(\sigma_A,\chi)\big|\ge \left(D_\ell+o(1)\right)(\log_2q)^{\ell+1}.$$
\end{theorem}
In recent years, the theory of Omega-theorems of $L$-functions has been well-developed due to the applications of the resonance method, which can date back to \cite{So} and \cite{V}. See \cite{Yangblms} for an overview.

To see whether the upper bounds or the Omega-results are the true maximum order, we show the following asymptotic formulae under  stronger conjectures.
\begin{corollary}\label{cor1}
    Assume Conjecture \ref{DYC} is true, then we have
    $$\max_{t\in[T,2T]}\big|\zeta^{(\ell)}(\sigma_A+\mmi t)\big|\sim \left(D_\ell+o(1)\right)(\log_2T)^{\ell+1},$$where $D_\ell$ is defined in Theorem \ref{thm3}.
\end{corollary}

\begin{corollary}\label{cor2}
    Assume Conjecture \ref{GSC} is true, then we have
     $$\max_{\chi\neq\chi_0({\rm mod}\; q)}\big|L^{(\ell)}(\sigma_A,\chi)\big|\sim \left(D_\ell+o(1)\right)(\log_2q)^{\ell+1},$$where $D_\ell$ is defined in Theorem \ref{thm3}.
\end{corollary}
It expected to show that, if $\sigma$ goes beyond the range in Theorem \ref{RH_upperBound}, the maximal order will also exceeds $(\log_2t)^{\ell+1}$. 
\begin{theorem}\label{thm5}
    Let $\frac12<\sigma<1-1/\log_2T$ and $A(T):=(1-\sigma)\log_2T$.  Then we have
  $$\max_{t\in[T,2T]}|\zeta^{(\ell)}(\sigma+\mmi t)|\gg \exp\bigg(\frac{\mme^{A(T)}}{A(T)}-A(T)\bigg) (\log_2T)^{\ell+1}.$$
\end{theorem}
This shows that, when $(1-\sigma)\log_2t\to\infty$ as $t\to\infty$, we  have $$\limsup_{t\to\infty}\frac{|\zeta^{(\ell)}(\sigma+\mmi t)|}{(\log_2t)^{\ell+1}}=\infty.$$
Similar convolution also holds for Dirichlet $L$-functions. 
\begin{theorem}\label{thm6}
   Let $\frac12<\sigma<1-1/\log_2q$ and $A(q):=(1-\sigma)\log_2q$.   Then we have
   $$\max_{\chi\neq\chi_0({\rm mod}\; q)}|L^{(\ell)}(\sigma,\chi)|\gg \exp\bigg(\frac{\mme^{A(q)}}{A(q)}-A(q)\bigg)(\log_2q)^{\ell+1}.$$
\end{theorem}

This article is organized as follows. We will present some preliminary lemmas in \S\ref{sec2}. We will prove Theorems \ref{RH_upperBound}-\ref{thm6} separately in   \S \ref{sec3}-\ref{sec8}. We also mention the sketch of  proofs of Corollary \ref{cor1} and \ref{cor2} at the end of \S \ref{sec6}.

\section{Preliminary Lemmas}\label{sec2}
Firstly, for the properties of $\Psi(x , y)$ and the Dickman function $\rho(u)$, we have the following lemma.
\begin{lemma}\label{HTDickman}
Let $x \geq  y \geq  2$ be real numbers, and put $ u = \frac{\log x}{\log y} $. For any fixed $\epsilon > 0$ the asymptotic formula
\begin{align}\label{JianJin}
    \Psi(x , y) = x \rho (u) \left( 1 +  O \left( \frac{\log(u+1)}{\log y}    \right)  \right),
\end{align}
holds uniformly in the range $1 \leq  u \leq  \exp\left(  \left(   \log y \right)^{\frac{3}{5} - \epsilon}  \right)$.\, The weaker relation
\begin{align}\label{JianJin2}
\log \frac{\Psi(x,y)}{x} = \left( 1+ O\left(\exp(-(\log u)^{\frac 35 
-\epsilon})\right)\right)\log \rho(u),
\end{align}
holds uniformly in the range $1\leq  u\leq  y^{1-\epsilon}$.  And as $u \to \infty$,
\begin{align}\label{decay}
    \log \rho (u) = - u \left(\log u + \log_2 (u+2) - 1 + O\left( \frac{\log_2 (u+2)}{\log (u+2)} \right)\right).
\end{align}
\end{lemma}
\begin{proof}
    See  Theorem $1.1$, $1.2$, and Corollary $2.3$ of \cite{HT}.
\end{proof}

We have the following conditional approximation formula for $\log \zeta(\sigma+\mmi t)$, which is adapted from Lemma 1 of \cite{GS}.

\begin{lemma}\label{approx_logZeta}
Assume RH. Let $y\geq  2$ and $t \geq  y+3$. For $\frac{1}{2} < \sigma \leq  1$, we have
$$
\log \zeta(\sigma+\mmi t)= \sum_{n=2}^{[y]}
\frac{\Lambda(n)}{n^{\sigma+\mmi t} \log n} + O\Big(
\frac{\log t}{(\sigma_1-\frac{1}{2})^2}y^{\sigma_1-\sigma}\Big),
$$
where we put $\sigma_1 = \min(\frac{1}{2} +\frac{1}{\log y},
\frac{\sigma}{2} + \frac{1}{4})$.
 
\end{lemma}
We have the following unconditional approximation formula for $\zeta^{(\ell)}(\sigma+ \mmi t)$. The constant 6.28 can be replaced by any positive number smaller than $2\pi$ (see \cite[Lemma 2]{HaL}).

\begin{lemma}
 Let  $\sigma_0 \in (0, 1)$ be fixed. If $T$ is sufficiently large,
then uniformly for $\epsilon >0$, $t \in [T, \, 6.28T]$, $\sigma \in[ \sigma_0 + \epsilon, \, \infty)$ and all positive integers $\ell$, we have
\begin{equation}\label{approx}
   (-1)^{\ell} \zeta^{(\ell)}(\sigma+ \mmi t) = \sum_{n \leq  T} \frac{(\log n)^{\ell}}{n^{\sigma+\mmi t}} + O\Big(  \frac{\ell !}{\epsilon^{\ell}}\cdot T^{-\sigma+\epsilon}\Big),
\end{equation}
where the implied constant in big $O(\cdot)$ only depends on $\sigma_0$ .
\end{lemma}

\begin{proof}
    See  \cite[Lemma 1]{DY}.
\end{proof}

The following approximation follows from  partial summation and  P\'{o}lya-Vinogradov inequality.
\begin{lemma}
Let $\chi \neq \chi_0\pmod q$ and $\ell \leq  \log N$. Then we have
\begin{align}\label{approxD}
    L^{(\ell)}(\sigma, \chi) = \sum_{k \leq  N} \frac{\chi(k) (-\log k)^{\ell}}{k^\sigma} + O\left( \frac{\sqrt{q} \log q (\log N)^{\ell}}{N} \right).
\end{align}
\end{lemma}
\begin{proof}
   See \cite[Thm 9.18]{M2} or \cite[Sec. 5, Eq. 16]{Yangblms}.
\end{proof}
We also need the following estimate for sums concerning primes.
\begin{lemma}\label{lem5}
    There exists an absolute constant $C$ such that for $(1-\sigma)\log x\ge\frac12$ we have 
    $$\sum_{p\le x}\frac{1}{p^\sigma}\ge \sigma\log_2x+\frac{x^{1-\sigma}}{(1-\sigma)\log x}+C.$$
\end{lemma}
\begin{proof}
   See \cite[Lemma 6]{BS18}.
\end{proof}

\section{Proof of Theorem \ref{RH_upperBound}}\label{sec3}

Let $x_1 = \exp \left( (\log_2 T)^2 \right)$, $x_2 = T $, and  $y_j =  \log^2 T\log^{2} x_j\, (\log_2 T)^{12}$ for $j = 1,\, 2.$ We have $\log y_1 \sim 2 \log_2 T,$ as $ T \to \infty.$
Taking $\epsilon=(\log_2 T)^{-1}$ and $1-A/\log_2T<\sigma<1$ in \eqref{approx}, we have
$$
(-1)^{\ell}\zeta^{(\ell)}(\sigma+\mmi t)
= \sum_{k\leq  T}\frac{(\log k)^{\ell}}{k^{\sigma+\mmi t}} + O\big((\log_2 T)^{\ell}\big),  \quad \forall t \in [2T,\, 5T]\,.
$$
Write
\begin{align*}
 \sum_{k \leq  x_2} \frac{ (\log k)^{\ell}}{k^{\sigma+\mmi t}} =\bigg( \sum_{k \leq  y_1 }  + \sum_{y_1 < k  \leq  x_2} \bigg)\frac{ (\log k)^{\ell}}{k^{\sigma+\mmi t}}.
\end{align*}
For the first summation, we have
\begin{align*}
\bigg| \sum_{k \leq  y_1 } \frac{ (\log k)^{\ell}}{k^{\sigma+\mmi t}} \bigg| \leq  \sum_{k \leq  y_1 } \frac{ (\log k)^{\ell}}{k^{\sigma}}& =\bigg( \int_0^1 u^\ell \mme^{2Au}\dif u  + o(1) \bigg) \left(\log y_1 \right)^{\ell + 1}\\&= \bigg( \int_0^1 u^\ell \mme^{2Au} \dif u + o(1) \bigg) \left(2\log_2T \right)^{\ell + 1}.
\end{align*}
For the second summation,  we have
\begin{align*}
 &\sum_{y_1 < k  \leq  x_2} \frac{ (\log k)^{\ell}}{k^{\sigma+\mmi t}} = \int_{y_1}^{x_2}\frac{(\log x)^{\ell}}{x^{\sigma}}\dif \bigg(\sum_{k\leq  x}  \frac{1}{k^{\mmi t}}\bigg)\\
 &=\frac{ (\log x_2)^{\ell}}{x_2^{\sigma}} \sum_{k\leq  x_2}  \frac{1}{k^{\mmi t}}   -  \frac{ (\log y_1)^{\ell}}{y_1^{\sigma}} \sum_{k\leq  y_1}  \frac{1}{k^{\mmi t}}    +  \int_{y_1}^{x_2} \sum_{k\leq  x}  \frac{1}{k^{\mmi t}}  \frac{\dif}{\dif x}   \bigg( \frac{-(\log x)^{\ell}}{x^{\sigma}}\bigg)  \dif x.  
\end{align*}
Using \eqref{Upper_BySmooth}, we have
\begin{align*}
 \bigg| \frac{ (\log x_2)^{\ell}}{x_2^{\sigma}} \sum_{k\leq  x_2}  \frac{1}{k^{\mmi t}} \bigg| \ll \frac{(\log T)^{\ell}}{T^{\sigma}} \Psi (T, \log^2 T (\log_2 T)^{20}).
\end{align*}
And by Lemma \ref{HTDickman}, we have
\begin{align*}
    \Psi (T, \log^2 T (\log_2 T)^{20})\ll T\rho (u),\quad \rho (u) \ll \mme^{-u\log u},
\end{align*}
with $u\ll \frac{\log T}{2\log_2T}$, we have
\begin{align*}
    \Psi (T, \log^2 T (\log_2 T)^{20})&\ll T \cdot \mme^{-\frac{\log T}{2\log_2T}\big(\log_2T+O(\log_3T)\big)}\\
    &=T \cdot (T^{-\frac{1}{2}+O(\frac{\log_3T}{\log_2T})}).
\end{align*}
Then,
\begin{align*}
    \bigg| \frac{ (\log x_2)^{\ell}}{x_2^{\sigma}} \sum_{k\leq  x_2}  \frac{1}{k^{\mmi t}}  \bigg| \ll T^{1-\sigma}(\log T)^{\ell}T^{-\frac{1}{2}+o(1)} \ll o(1) \cdot \left(\log_2 T \right)^{\ell+1}.
\end{align*}
Clearly, we have 
\begin{align*}
    \bigg|   \frac{ (\log y_1)^{\ell}}{y_1^{\sigma}} \sum_{k\leq  y_1}  \frac{1}{k^{\mmi t}}  \bigg| \leq  y_1^{1-\sigma}(\log y_1)^{\ell}\leq   o(1) \cdot \left(\log_2 T \right)^{\ell+1}.
\end{align*}
We split the integration from $y_1$ to $x_1$ into two parts:
\begin{align*}
    \int_{y_1}^{x_2}& \sum_{k\leq  x}  \frac{1}{k^{\mmi t}}  \frac{\dif}{\dif x}   \bigg( \frac{-(\log x)^{\ell}}{x^{\sigma}}\bigg)  \dif x\\&=\int_{y_1}^{x_1} \sum_{k\leq  x}  \frac{1}{k^{\mmi t}} \frac{\dif}{\dif x}   \bigg( \frac{-(\log x)^{\ell}}{x^{\sigma}}\bigg)  \dif x+\int_{x_1}^{x_2} \sum_{k\leq  x}  \frac{1}{k^{\mmi t}} \frac{\dif}{\dif x}   \bigg( \frac{-(\log x)^{\ell}}{x^{\sigma}}\bigg)  \dif x.
\end{align*}
For the first integration, by \eqref{Approx_BySmooth}, we have
\begin{align*}
 \left|   \int_{y_1}^{x_1} \sum_{k\leq  x}  \frac{1}{k^{\mmi t}}  \frac{\dif}{\dif x}   \left( \frac{-(\log x)^{\ell}}{x^{\sigma}}\right)  \dif x \right| &\leq  \int_{y_1}^{x_1} \left( 1 + o(1) \right) \Psi(x, y_1)  \frac{\dif}{\dif x}   \left( \frac{-(\log x)^{\ell}}{x^{\sigma}}\right)  \dif x.
\end{align*}
Let
    $u={\log x}/{\log y_1}$, by Lemma \ref{HTDickman}, we have
\begin{align*}
    \bigg|   \int_{y_1}^{x_1} &\sum_{k\leq  x}  \frac{1}{k^{\mmi t}}  \frac{\dif}{\dif x}  \bigg( \frac{-(\log x)^{\ell}}{x^{\sigma}}\bigg)  \dif x \bigg|\\&\leq  \int_{1}^{\frac{\log x_1}{\log y_1}}(1+o(1))x\rho (u)\frac{\sigma(\log x)^{\ell}-\ell(\log x)^{\ell-1}}{x^{\sigma+1}}x\log y_1\dif u\\
    &\leq  \int_{1}^{\infty}(1+o(1))\rho (u)x^{1-\sigma}\left(\sigma(\log x)^{\ell}-\ell(\log x)^{\ell-1}\right)\log y_1\dif u.
    \end{align*}
Since $\ell$ is fixed, the above is
    \begin{align*}
    &\leq   (\log y_1)^{\ell+1}\int_{1}^{\infty}(1+o(1))\rho (u)\mme^{2Au}u^{\ell}\dif u\\
   & = (2\log_2T)^{\ell+1}\Big(\int_{1}^{\infty}u^{\ell}\rho(u)\mme^{2Au}\dif u+o(1)\Big).
\end{align*}
For the second integration, by \eqref{Approx_BySmooth}  and Lemma \ref{HTDickman}, we have
\begin{align*}
 \bigg|   \int_{x_1}^{x_2}& \sum_{k\leq  x}  \frac{1}{k^{\mmi t}}  \frac{\dif}{\dif x}   \bigg( \frac{-(\log x)^{\ell}}{x^{\sigma}}\bigg)  \dif x \bigg|\\ &\leq  \int_{x_1}^{x_2} \left( 1 + o(1) \right) \Psi(x, y_2)  \frac{\dif}{\dif x}   \bigg( \frac{-(\log x)^{\ell}}{x^{\sigma}}\bigg)  \dif x  \\
 &\leq  \bigg(\int_{\frac{\log x_1}{\log y_2}}^{\infty}u^{\ell}\rho(u)\mme^{2Au}\dif u+o(1)\bigg)\cdot (2\log_2T)^{\ell+1}.
\end{align*}
Since $ {\log x_1}/{\log y_2} \to \infty$ as $T \to \infty$, we have
\begin{align*}
    \bigg|   \int_{x_1}^{x_2} \bigg(\sum_{k\leq  x}  \frac{1}{k^{\mmi t}} \bigg) \frac{\dif}{\dif x}   \bigg( \frac{-(\log x)^{\ell}}{x^{\sigma}}\bigg)  \dif x \bigg| \leq  o(1) \cdot (\log_2T)^{\ell+1}.
\end{align*}
At last we obtain
\begin{align*}
   | \zeta^{(\ell)}\left(\sigma+\mmi t\right)| &\leq  \bigg( \int_0^1 \mme^{2Au}u^\ell \dif u +\int_1^{\infty} \mme^{2Au} u^{\ell} \rho (u) \dif u  + o(1) \bigg)  \left(2\log_2 T \right)^{\ell+1}\\  & =  \left( 2^{\ell + 1}C_{\ell } + o(1) \right) \left(\log_2 T \right)^{\ell + 1}.
\end{align*}
This completes the proof of Theorem \ref{RH_upperBound} combining with $t\in[2T,5T]$.

\section{Proof of Theorem \ref{GRH_upperBound}}\label{sec4}
Let $x_1 = \exp \left( (\log_2 q )^2 \right)$, $x_2 =  q^{\frac{3}{4}} $ , and  $ y_j = \log^2q\log^{2} x_j \, (\log_2 q )^{12}$ for $j = 1,\, 2.$  We will use the  approximation formula \eqref{approxD} for $L^{(\ell)}(1, \chi) $  and other steps are the same as the proof of Theorem  \ref{RH_upperBound}.

\section{Proof of Theorem \ref{thm3}}\label{sec5}
\subsection{Log-type GCD sums}
Let $T$ be large. Let $w = \pi(y)$. Define $y$, $b$ and $\PP(y, b)$  as  follows
$$ y = \frac{\log T}{3  (\log_2 T)^{2e^A+1}}    \,,\quad b = \lfloor(\log_2 T)^{2e^A+1}\rfloor\,, \quad    \PP(y, b) = \prod_{p \leq  y} p^{b-1} = \prod_{j = 1}^w p_j^{b-1}.$$
We show the following lower bound of the log-type GCD sums, which is the main ingredient of the proof of Theorem \ref{RH_upperBound}.
\begin{proposition}\label{maxRatio}
  Let $\sigma=\sigma_A=1-A/(\log_2T)$.  As $T \to \infty$,  uniformly for all positive numbers $\ell \leq  (\log_3 T) / (\log_4 T)$, we have
\begin{equation*}
\sup_{r} \Big|\sum_{mk = n\leq  \sqrt T} \frac{r(m)\overline{r(n)}}{ k^\sigma} (\log k)^{\ell}
\Big| \Big/ \Big(\sum_{n\leq  \sqrt T }|r(n)|^2\Big) 
\geq  \big(D_\ell + o\left(1\right)\big)\left(\log_2 T \right)^{\ell+1},
\end{equation*}
where the supreme is taken over all functions $r :\, \mathbb N \to \mathbb C$ satisfying that the denominator is not equal to zero, when the parameter $T$ is given. 
\end{proposition}
\begin{proof}
Let $\M:=\{n\in\mathbb N\,|\, n|\PP(y,b)\}$. Choose   $r$ to be the characteristic function of $\M$. Then
\begin{align*}
&\frac{1}{\sum_{n\leq  \sqrt T }|r(n)|^2}\left|\sum_{mk = n\leq  \sqrt T} \frac{r(m)\overline{r(n)}}{ k^\sigma} (\log k)^{\ell}
\right|  \\
=&\frac{1}{|\M|}\sum_{n\in\M\atop k|n}\frac{(\log k)^\ell}{k^\sigma}=\frac{1}{|\M|}\sum_{k\in\M }\frac{(\log k)^\ell}{k^\sigma}\sum_{n\in\M\atop k|n}1.\end{align*}
Since each term is positive, we have
$$\frac{1}{|\M|}\sum_{k\in\M}\frac{(\log k)^\ell}{k^\sigma}\sum_{n\in\M\atop k|n}1\ge\frac{1}{|\M|}\sum_{k\le(\log_2T)^{\log_2T},\;k\in \mms(y) \atop \Omega(k)\le{(\log_2T)^{2e^A+1}}/{(\log_3T)} }\frac{(\log k)^\ell}{k^\sigma}\sum_{n\in\M\atop k|n}1.$$
Denote $\alpha:={(\log_2T)^{2e^A+1}}/{(\log_3T)}$. Since $k\le(\log_2T)^{\log_2T}$, $k\in \mms(y)$ and $\Omega(k)\le\alpha $ imply
$$\frac{1}{|\M|}\sum_{n\in\M\atop k|n}1\ge1-\frac{2}{\log_3T}=1+o(1),$$
we have 
$$\frac{1}{|\M|}\sum_{k\le(\log_2T)^{\log_2T},\;k\in \mms(y) \atop \Omega(k)\le\alpha }\frac{(\log k)^\ell}{k^\sigma}\sum_{n\in\M\atop k|n}1=(1+o(1))\sum_{k\le(\log_2T)^{\log_2T},\;k\in \mms(y) \atop \Omega(k)\le\alpha}\frac{(\log k)^\ell}{k^\sigma}.$$
Write
$$\sum_{k\le(\log_2T)^{\log_2T},\;k\in \mms(y) \atop \Omega(k)\le\alpha}\frac{(\log k)^\ell}{k^\sigma}=\bigg(\sum_{k\le(\log_2T)^{\log_2T} \atop  k\in \mms(y)}-\sum_{k\le(\log_2T)^{\log_2T},\;k\in \mms(y) \atop \Omega(k)>\alpha}\bigg)\frac{(\log k)^\ell}{k^\sigma}.$$
The second sum is $o((\log_2T)^{\ell+1})$. Combining the above, we have
\begin{align}\label{GCDsubset}\frac{1}{|\M|}\sum_{k\in\M}\frac{(\log k)^\ell}{k^\sigma}\sum_{n\in\M\atop k|n}1\ge (1+o(1))\sum_{k\le(\log_2T)^{\log_2T} \atop  k\in \mms(y)} \frac{(\log k)^\ell}{k^\sigma}.\end{align}
 We split the sum into two parts as follows
\begin{align*}
\sum_{k\le(\log_2T)^{\log_2T} \atop  k\in \mms(y)} \frac{(\log k)^\ell}{k^\sigma}=   \bigg(\sum_{k\le y}+\sum_{y<k\le(\log_2T)^{\log_2T} \atop  k\in \mms(y)}\bigg)\frac{(\log k)^\ell}{k^\sigma}  = S_1 + S_2\,.
\end{align*}
The first sum is
\begin{align*}
  S_1 =\sum_{k \leq  y} \frac{(\log k)^{\ell}}{k^\sigma} = \left( \int_0^1 \mme^{Au}u^\ell \dif u + o(1) \right) \left(\log_2 T \right)^{\ell+1}. 
\end{align*}
By partial summation, writing $R=(\log_2T)^{\log_2T}$, the second sum is 
\begin{align}\label{S2}
  S_2 = \frac{(\log R)^{\ell}}{R^\sigma} \Psi (R, y) - \frac{(\log y)^{\ell}}{y^\sigma} \Psi (y, y)  - \int_y^{\RR} \frac{\dif}{\dif x}   \left( \frac{(\log x)^{\ell}}{x^\sigma}\right) \Psi(x, y) \dif x.
\end{align}
By \eqref{JianJin},  we have 
\begin{align}\label{Dickman}
   \Psi(x, y) = x\, \rho \left ( \frac{\log x}{\log y} \right) \left( 1 + O\left( \frac{\log_4 T}{\log_2 T} \right)\right)\,, \quad \text{uniformly for} \,\, y \leq  x \leq  R.
\end{align}
Applying \eqref{Dickman} into \eqref{S2}, and using  \eqref{decay}, we obtain 
\begin{align*}
   S_2 =  \left ( \int_1^{\infty}  \mme^{Au}u^{\ell} \rho (u) \dif u  + o(1) \right)  \left(\log_2 T \right)^{\ell+1}.
\end{align*}
Thus we get
\begin{align*}
 S_1 + S_2  = \left ( \int_0^{\infty}  \mme^{Au}u^{\ell} \rho (u) \dif u  + o(1) \right)  \left(\log_2 T \right)^{\ell+1}.
\end{align*}
Inserting into \eqref{GCDsubset}, we complete the proof of Proposition \ref{maxRatio}.
\end{proof}

 \subsection{End of the proof}
By \cite[page 496]{DY}, we have
\begin{align*}
\max_{T\leq  t\leq  2T}\left|\zeta^{(\ell)}\Big(\sigma+\mmi t\Big)\right| & \geq  \big(1 + O(T^{-1})\big)
  \Big|\sum_{mk = n\leq  \sqrt T} \frac{r(m)\overline{r(n)}}{ k} (\log k)^{\ell}
\Big| \Big/ \Big(\sum_{n\leq  \sqrt T }|r(n)|^2\Big) \\
&+\,O\Big( T^{-\frac{3}{2}}\, (\log T)^{\ell +1} \Big)  + O\Big(  (\log_2 T)^{\ell}\Big). \end{align*}   
By  Proposition \ref{maxRatio}, we finish the proof of Theorem \ref{thm3}.

\section{Proof of Theorem \ref{thm4} and Corollary \ref{cor1}-\ref{cor2}}\label{sec6} 
\subsection{Proof of Theorem \ref{thm4}}

In order to use Soundararajan's resonance method \cite{So} to produce  extreme values, we define $V_2(q)$ and $V_1(q)$ as follows  
\begin{align*}
    V_2(q):\,  = \sum_{\chi \neq \chi_0}  (-1)^{\ell}  L^{(\ell)}(\sigma, \chi; N) \left|R_{\chi}\right|^2, \quad
    V_1(q):\,  = \sum_{\chi\neq \chi_0}  \left|R_{\chi}\right|^2,
\end{align*}
where $L^{(\ell)}(\sigma, \chi; N)$ and the resonator  $R_{\chi}$ are defined by 
\begin{align*}
      L^{(\ell)}(\sigma, \chi; N):\, = \sum_{k \leq  N} \frac{\chi(k) (-\log k)^{\ell}}{k^\sigma}, \quad
    R_{\chi}:\,  = \sum_{m \leq  M}\chi(m) r(m).
\end{align*}
We choose $T = q^{\frac{1}{2}}$, $N = q^{\frac{3}{4}}$, $M = q^{\frac{1}{4}}$ and let the function $r(n)$ be defined  as in the proof of Proposition \ref{maxRatio}. By  orthogonality of characters, we have
\begin{align}\label{V1Shangjie}
      V_1(q) \leq  \sum_{\chi}  \left|R_{\chi}\right|^2 \leq  \phi(q) \sum_{m \leq  M} r(m). 
\end{align}
By Cauchy’s inequality, we have \[\left|R_{\chi_0}\right|^2 \leq  M \sum_{m \leq  M} r(m).\]
Thus we can bound  $\left| L^{(\ell)}(\sigma, \chi_0; N)\right|\cdot \left|R_{\chi_0}\right|^2$ by
\begin{align*}
    \leq  (\log q)^{\ell + 1}  M \sum_{m \leq  M} r(m).
\end{align*}
Above upper bound together with the orthogonality of characters gives that
\begin{align}\label{V2JianJin}
    V_2(q) = \phi(q) \sum_{mk = n \leq  M}  \frac{(\log k)^{\ell} r(m)r(n)}{k^\sigma}  +  O\left( (\log q)^{\ell + 1}    \right) \cdot  M \sum_{m \leq  M} r(m).
\end{align}
Combining \eqref{V2JianJin} with \eqref{V1Shangjie}, we have 
\begin{align}\label{ChiRatio}
     \max_{\chi \neq \chi_0} &\left|L^{(\ell)}(\sigma, \chi; N) \right| \geq  \left| \frac{V_2(q)}{V_1(q)} \right|\\& = \left(
    \sum_{mk = n \leq  M}  \frac{(\log k)^{\ell} r(m)r(n)}{k^\sigma}\right) \Big/\left( \sum_{m \leq  M} r(m) \right) + O\left( \left(\log q\right)^{\ell+1} \right)\cdot q^{-\frac{3}{4} }.\nonumber
\end{align}
By \eqref{ChiRatio}, \eqref{approxD} and  Proposition \ref{maxRatio}, we obtain Theorem \ref{thm4}.
\subsection{Proofs of Corollary \ref{cor1} and \ref{cor2}}
The upper bound of Corollary \ref{cor1} follows from taking $x_1 = \exp \left( (\log_2 T)^2 \right)$, $x_2 = T $, and  $y_j  = (\log T +\log^2 x_j) (\log_2 T)^A $ for $j = 1,\, 2.$ in the proof of Theorem  \ref{RH_upperBound}. The lower bound is exactly Theorem \ref{thm3}.

The upper bound of Corollary \ref{cor2} follows from taking $x_1 = \exp \left( (\log_2 q )^2 \right)$, $x_2 =  q^{\frac{3}{4}} $ , and  $ y_j = (\log q + \log^2 x_j) (\log_2 q )^A  $ for $j = 1,\, 2$ in the proof of Theorem  \ref{GRH_upperBound}.   The lower bound is exactly Theorem \ref{thm4}.

\section{Proof of Theorem \ref{thm5}}\label{sec7}
Take $y=\log T/(3\log_2T)$, $b=\lfloor \log_2T\rfloor$ and $A(T)=(1-\sigma)\log_2T$. Without loss of generality, we assume $A(T)\to\infty$ as $T\to\infty$. Define
$$\mathcal P(y,b):=\prod_{p\le y}p^{b-1},\;\;\mathcal M:=\{n\in\mathbb N:\;n|\mathcal P(y,b)\}.$$
By \cite[page 496]{DY}, we have
\begin{align*}
\max_{T\leq  t\leq  2T}\left|\zeta^{(\ell)}\Big(\sigma+\mmi t\Big)\right| & \gg  
\sup_r  \Big|\sum_{mk = n\leq  \sqrt T} \frac{r(m)\overline{r(n)}}{ k^\sigma} (\log k)^{\ell}
\Big| \Big/ \Big(\sum_{n\leq  \sqrt T }|r(n)|^2\Big) ,
\end{align*}  
where the supreme is taken over all functions $r :\, \mathbb N \to \mathbb C$ satisfying that the denominator is not equal to zero.
By the definition of $\mathcal M$, the RHS is
$$\ge \frac{1}{|\M|}\sum_{n\in\M\atop k|n}\frac{(\log k)^\ell}{k^\sigma}.$$ Thus to prove Theorem \ref{thm5}, it suffice to show the following proposition.
\begin{proposition}\label{prop2}
 Let $\ell\in\mathbb N$ be fixed and $T \to \infty$, then we have
\begin{equation*}
\frac{1}{|\M|}\sum_{n\in\M\atop k|n}\frac{(\log k)^\ell}{k^\sigma}
\gg \exp\bigg(\frac{\mme^{A(T)}}{A(T)}-A(T)\bigg) \left(\log_2 T \right)^{\ell+1}.
\end{equation*}

\end{proposition}
\begin{proof}
Firstly we have
$$\frac{1}{|\M|}\sum_{n\in\M\atop k|n}\frac{(\log k)^\ell}{k^\sigma}\ge\frac{1}{|\M|}\sum_{n\in\M\atop k|n,\; k>\sqrt y}\frac{(\log k)^\ell}{k^\sigma}\ge\frac{(\log \sqrt y)^\ell}{|\M|}\bigg(\sum_{n\in\M\atop k|n}-\sum_{n\in\M\atop k|n,\; k\le\sqrt y}\bigg)\frac{1}{k^\sigma}.$$
For the second sum, we have
\begin{align}\label{eq17}\frac{1}{|\M|}\sum_{n\in\M\atop k|n,\; k\le\sqrt y}\frac{1}{k^\sigma}=\sum_{k\in\M\atop k\le\sqrt y}\frac{1}{k^\sigma}\bigg(\frac{1}{|\M|}\sum_{n\in\M\atop k|n}1\bigg)\le\sum_{k\le\sqrt y}\frac{1}{k^\sigma}\le \frac{y^{\frac12(1-\sigma)}}{(1-\sigma)\log y}.
\end{align}
For the first sum, we have
\begin{align}\label{eq18}\frac{(\log \sqrt y)^\ell}{|\M|}\sum_{n\in\M\atop k|n}\frac{1}{k^\sigma}\gg_\ell(\log y)^\ell\frac{1}{|\M|}\sum_{n\in\M\atop k|n}\frac{1}{k^\sigma}.\end{align}
By G\'al's indentity (see Eq. 4.4 of \cite{BS18} for example), we have
$$\frac{1}{|\M|}\sum_{n\in\M\atop k|n}\frac{1}{k^\sigma}=\prod_{p\le y}\sum_{v=0}^{b-1}\bigg(1-\frac{v}{b}\bigg)\frac{1}{p^{v\sigma}}=\exp\bigg((1+O(b^{-1}))\sum_{p\le y}\frac{1}{p^{\sigma}}+O(1)\bigg).$$
By Lemma \ref{lem5}, this is
\begin{align*}\gg(\log y)^\sigma\exp\bigg(\frac{y^{1-\sigma}}{(1-\sigma)\log y}\bigg)&=\log y\exp\bigg(\frac{y^{1-\sigma}}{(1-\sigma)\log y}\bigg)/(\log y)^{1-\sigma}\\&\ge\log y\exp\bigg(\frac{y^{1-\sigma}}{(1-\sigma)\log y}\bigg)/y^{1-\sigma}.\end{align*}
This combined with the definition of $y$, Eq. \eqref{eq17} and \eqref{eq18}, completes the proof of Proposition \ref{prop2}.
\end{proof}

\section{Proof of Theorem \ref{thm6}}\label{sec8}
Using \eqref{ChiRatio}, taking $y=\log q/(3\log_2q)$, $b=\lfloor \log_2q\rfloor$ and $A(q)=(1-\sigma)\log_2q$, then the proof of Theorem \ref{thm6} is the same with that of Theorem \ref{thm5}.
\section*{Acknowledgements}
The authors would like to thank Professor Zhonghua Li for helpful discussion and encouragement.

\end{document}